\documentclass[journal,twoside,web]{ieeecolor}
\usepackage{generic}
\usepackage{cite}
\usepackage{amsmath,amssymb,amsfonts}
\usepackage{graphicx}
\usepackage{textcomp}
\usepackage{bm}
\usepackage{subfigure}
\usepackage{multirow}
\usepackage{caption}
\usepackage{makecell}
\usepackage{cite}
\usepackage{color}
\usepackage{booktabs}
\usepackage{epstopdf}
\usepackage{bm}

\usepackage{algorithm}
\usepackage{algorithmicx}
\usepackage{algpseudocode}

\newtheorem{theorem}    {Theorem}
\newtheorem{definition} {Definition}
\newtheorem{corollary}  {Corollary}
\newtheorem{lemma}      {Lemma}
\newtheorem{remark}     {Remark}

\def\BibTeX{{\rm B\kern-.05em{\sc i\kern-.025em b}\kern-.08em
    T\kern-.1667em\lower.7ex\hbox{E}\kern-.125emX}}
\begin{document}
\title{Accelerating the Convergence Rate of Consensus for Second-Order Multi-Agent Systems by Memory Information}
\author{Jiahao~Dai, Jing-Wen~Yi$^*$, \IEEEmembership{Member, IEEE} and Li~Chai, \IEEEmembership{Member, IEEE}
\thanks{This work was supported by the National Natural Science Foundation of China (grant 62173259, 62176192, 61625305, and 61701355). $^*$Corresponding author: Jing-Wen Yi.}
\thanks{Jiahao Dai and Jing-Wen Yi are with the Engineering Research Center of Metallurgical Automation and Measurement Technology, Wuhan University of Science and Technology, Wuhan 430081, China. E-mail: daijiahao@wust.edu.cn; yijingwen@wust.edu.cn.}
\thanks{Li Chai is with the College of Control Science and Engineering, Zhejiang University, Hangzhou 310027, China. E-mail: chaili@zju.edu.cn.}}

\maketitle

\begin{abstract}
This paper utilizes the agent's memory in accelerated consensus for second-order multi-agent systems (MASs).
In the case of one-tap memory, explicit formulas for the optimal consensus convergence rate and control parameters are derived by applying the Jury stability criterion.
It is proved that the optimal consensus convergence rate with one-tap memory is faster than that without memory.
In the case of M-tap memory,
an iterative algorithm is given to derive the control parameters to accelerate the convergence rate.
Moreover, the accelerated consensus with one-tap memory is extended to the formation control, and
the control parameters to achieve the fastest formation are obtained.
Numerical examples further illustrate the theoretical results.
\end{abstract}

\begin{IEEEkeywords}
Multi-agent systems, second-order, memory, consensus, convergence rate.
\end{IEEEkeywords}

\section{Introduction}
During the past few decades, the consensus problem has received significant attention due to its wide application in formation control \cite{kamel2020formation}, cooperative control \cite{knorn2015overview}, flocking \cite{olfati2006flocking}, and so on.
The main purpose of consensus is to make the state of all agents reach a common value by designing a control protocol that utilizes the local information between an agent and its neighbors.

Consensus of second-order MASs has been extensively studied, since many vehicle dynamics can be linearized as a double integrator \cite{ren2007distributed,ren2008distributed,ren2011distributed}.
Han et al. \cite{han2018second} used algebraic graph theory and matrix theory to analyze the consensus problem of second-order MASs, and established a connection between impulsive control methods and network topologies.
Fu et al. \cite{fu2018consensus} considered the consensus control of second-order MASs constrained by both speed and input on a directed network, and designed a consensus controller without relying on global information.
Shi et al. \cite{shi2020scaled} focused on the scale tracking consensus problem of discrete-time second-order MASs under random packet dropouts, and established sufficient conditions related to the topological structure and the successful probabilities of information transmission.
Yang et al. \cite{yang2020sampled} studied the consensus of second-order MASs, and investigated the positive effects of time-varying communication delays.

Driven by various applications, accelerated consensus has been an active research area in second-order MASs.
Some efforts are performed in continuous-time systems \cite{zhu2011consensus,fu2017finite,difilippo2021maximizing}.
Zhu et al. \cite{zhu2011consensus} focused on consensus speed for continuous-time second-order MASs, and showed that the maximum consensus speed is determined by the smallest nonzero and the largest eigenvalues of the graph Laplacian.
Fu et al. \cite{fu2017finite} designed a finite-time consensus controller which involve a saturation function for continuous-time second-order MASs, and provided some sufficient criteria for finite-time leaderless and leader-following consensus by using the Lyapunov stability theory.
Difilippo et al. \cite{difilippo2021maximizing} considered a leaderless consensus protocol for continuous-time second-order MASs, and obtained the maximizing convergence speed for a class of digraphs having a directed spanning tree.

For discrete-time second-order MASs, the accelerated consensus is considered in \cite{2011Network,eichler2014closed,eichler2017optimal1}.
You et al. \cite{2011Network} derived the optimal consensus convergence rate for discrete-time second-order MASs, by analyzing the eigenvalues of the closed-loop matrix in the complex plane.
Eichler et al. \cite{eichler2014closed} optimized the convergence rate for discrete-time second-order MASs, and showed that the global optimum of the convergence rate can be obtained by solving an LMI problem.
Next, in \cite{eichler2017optimal1}, they also optimized the convergence rate under given damping constraints for discrete-time second-order MASs, and designed a combined bisection grid search algorithm to solve the global optimum.
It is worth noting that the controllers in \cite{2011Network,eichler2014closed,eichler2017optimal1} are memoryless, i.e., only the current states of the agent or its neighbors are used.



Utilizing the agent's memory information is an effective way to accelerate the consensus convergence rate, which has been widely applied in discrete-time first-order MASs \cite{8716798,9117158,Yi2021convergence,9763023}.
Previous work \cite{Yi2021convergence} considered a consensus protocol with node memory for first-order MASs, and derive explicit formulas for the optimal convergence rate and control parameters.
Besides, we also designed a general consensus protocol with memory for first-order MASs in \cite{9763023},
and proved that the memory information of the agent's neighbors is not effective for further acceleration of the convergence rate in the worst-case scenario.


In this paper, we consider the convergence acceleration for discrete-time second-order MASs by the agent's memory information.
The consensus problem of second-order MASs is transformed into the simultaneous stabilization problem of multiple subsystems by applying the graph Fourier transform.
After presenting the necessary and sufficient condition for consensus, some analytical or numerical results are given.
The main contributions are summarized in the following.


(i) In the case of one-tap memory, explicit formulas for the optimal consensus convergence rate and control parameters are derived by utilizing the Jury stability criterion.
It is proved that the optimal consensus convergence rate with one-tap memory is faster than that without memory.

(ii) In the case of $M$-tap memory,
an iterative algorithm based on gradient descent is designed to derive the control parameters to accelerate the convergence rate.
It is found that the obtained convergence rate can be faster than the optimal convergence rate with one-tap memory.

(iii) The accelerated consensus with one-tap memory is extended to the formation control. The control parameters to achieve the fastest formation are given.

The rest of this paper is organized as follows.
In Section II, we review some basic results to be used in later sections, and formulate the accelerated consensus with $M$-tap memory.
In Section III, the consensus problem of second-order MASs is transformed into the simultaneous stabilization problem of multiple subsystems, and the necessary and sufficient condition for consensus is presented.
In Section IV, we derive the explicit formula of the optimal convergence rate for second-order MASs with one-tap memory, and design an algorithm to obtain the control parameters to accelerate the convergence rate.
In Section V, we extend the accelerated consensus with one-tap memory to the formation control.
In Section VI, three examples are given to verify the theoretical results.
Section VII concludes this paper.

\section{Preliminaries and Problem Formulation}

This section briefly reviews some basic results to be used in later sections, and formulates the accelerated consensus with $M$-tap memory.

\subsection{Preliminaries}

We use an undirected network $\mathcal{G\!=\!(V,E,A)}$ to describe the information interactions between agents, where
$\mathcal{V}\!=\!\{\text{v}_1,\text{v}_2,\cdots,\text{v}_N\}$ denotes the set of agents, $\mathcal{E \!\subseteq\! V\times V}$ denotes the set of edges, and $\mathcal{A}\!=\![a_{ij}]\!\in\! \mathbb{R}^{N\!\times\! N}$ denotes the adjacency matrix.
If the edge satisfies $(\text{v}_{i},\text{v}_{j})\in\mathcal{E}$, then $a_{ij}\!=\!a_{ji}\!>\!0$.
The set of neighbors of agent $i$ is represented by ${\mathcal{N}_i} \!= \!\left\{ {{\text{v}_j} \in \mathcal{V}:({\text{v}_i},{\text{v}_j}) \in \mathcal{E} } \right\}$.
Let $\mathcal{D}\!:=\!diag\{\text{d}_{1},\cdots,\text{d}_{N}\}$ be the degree matrix, where ${\text{d}_i}\!= \!\sum\nolimits_{j = 1}^N {{a_{ij}}}$ denotes the degree of $\text{v}_{i}$.
Then the Laplacian matrix of network $\mathcal{G}$ is defined as $\mathcal{L}\!=\!\mathcal{D\!-\!A}$.
For a connected network,
the Laplacian matrix has the spectral decomposition ${\mathcal{L}=\Omega \Lambda \Omega^{T}}$ \cite{2004Consensus}, where $\Lambda\!=\!diag\{\lambda_{1},\lambda_{2},\cdots,\lambda_{N}\}$ and $\Omega\!=\![\bm{w}_{1},\bm{w}_{2},\cdots,\bm{w}_{N}]\!\in\!\mathbb{R}^{N\times N}$ are unitary.
The eigenvalues of $\mathcal{L}$ can be sorted in ascending order $ {\lambda _1} \!<\! {\lambda _2} \!\le\!  \cdots  \!\le\! {\lambda _N}$, where
 $\lambda_1=0$ and its eigenvector is $\bm{w}_{1}\!=\!\frac{1}{\sqrt{N}}\bm{1}$.

 \begin{lemma} \cite{4066881}
Given a third-order polynomial $d(z) = a_0z^3+a_1z^2+a_2z+a_3, a_0>0$.
The Jury stability criterion  states that
the roots of $d(z) = 0$ are all within the unit circle if and only if
\[\left\{ \begin{array}{l}
\left| {{a_3}} \right| < {a_0},\\
{a_0} + {a_1} + {a_2} + {a_3} > 0,\\
{a_0} - {a_1} + {a_2} - {a_3} > 0,\\
\left| {a_3^2 - a_0^2} \right| < \left| {{a_1}{a_3} - {a_0}{a_2}} \right|.
\end{array} \right.\]
\end{lemma}

\begin{lemma} \cite{boyd2004convex} Given a matrix
\[Q = \left[ {\begin{array}{*{20}{c}}Q_1&Q_2\\Q_3&Q_4\end{array}}\right],\]
 with nonsigular $Q_1\in {\mathbb{R}^{\mu \times \mu}}$, $Q_2 \in {\mathbb{R}^{\mu \times n}}$, $Q_3 \in {\mathbb{R}^{n \times \mu}}$, and $Q_4 \in {\mathbb{R}^{n \times n}}$. Then
$\det Q= \det Q_1 \cdot \det (Q_4 - Q_3{Q_1^{ - 1}}Q_2)$.
\end{lemma}

\subsection{Problem Formulation}

The discrete-time second-order dynamics of agent $i$ is given by \cite{eichler2014closed,eichler2017optimal1,shi2020scaled,zou2021distributed}
\begin{equation}
\begin{aligned}
&{x_i}(k + 1) = {x_i}(k) + \tau{v_i}(k),\\
&{v_i}(k + 1) = {v_i}(k) + \tau{u_i}(k), i=1,\ldots,N,
\end{aligned}
\end{equation}
where ${x_i}(k) \in \mathbb{R}^n$ denotes the position, ${v_i}(k) \in \mathbb{R}^n$ denotes the velocity, $\tau$ denotes the sampling period and
${u_i}(k) \in \mathbb{R}^n$ is the control input.
For simplicity, only $n=1$ is considered.
The theoretical results of $n>1$ are still valid by using Kronecker products.

The control protocol with $M$-tap memory is designed as
\begin{equation}
\begin{aligned}
{u_i}(k) = & {\varepsilon _1}\!\sum\limits_{j \in {\mathcal{N}_i}} {{a_{ij}}({x_j}(k) \!-\! {x_i}(k)} ) \!+\! {\varepsilon _2}\!\sum\limits_{j \in {\mathcal{N}_i}} {{a_{ij}}({v_j}(k) \!-\! {v_i}(k)} ) \\
&+ \sum\limits_{m = 0}^M {{\theta _m}{v_i}(k \!-\! m)} ,
\end{aligned}
\end{equation}
where ${\varepsilon _1},{\varepsilon _2},{\theta _0},{\theta _1}, \ldots ,{\theta _M} \in \mathbb{R}$ are control parameters to be designed.
The initial states are set as $x_i(0)$ and $v_i(0)$.
Assume that
\begin{equation} \nonumber
\begin{aligned}
{v_i}(-M) =  \cdots  = {v_i}( - 1) = {v_i}(0).
\end{aligned}
\end{equation}

Note that the memoryless control protocol in existing literatures can be viewed as a special case in (2).
\\(i) If $\theta _0,\ldots,\theta _M=0$ and $\varepsilon_1\!=\!1$, then (2) becomes the protocol
\begin{equation} \nonumber
\begin{aligned}
{u_i}(k) =  \sum\limits_{j \in {\mathcal{N}_i}} {{a_{ij}}({x_j}(k) \!-\! {x_i}(k)} ) \!+\! {\varepsilon _2}\!\sum\limits_{j \in {\mathcal{N}_i}} {{a_{ij}}({v_j}(k) \!-\! {v_i}(k)} ),
\end{aligned}
\end{equation}
which has been proposed in \cite{eichler2014closed}.
\\(ii) If $\theta _0,\ldots,\theta _M\!=\!0$, then (2) becomes the protocol
\begin{equation} \nonumber
\begin{aligned}
{u_i}(k) =  {\varepsilon _1}\!\!\sum\limits_{j \in {\mathcal{N}_i}} {{a_{ij}}({x_j}(k) \!-\! {x_i}(k)} ) \!+\! {\varepsilon _2}\!\!\sum\limits_{j \in {\mathcal{N}_i}} {{a_{ij}}({v_j}(k) \!-\! {v_i}(k)} ),
\end{aligned}
\end{equation}
which has been proposed in \cite{eichler2017optimal1}.

\begin{definition}
Denote \[\bar x = \frac{1}{N}\sum\nolimits_{j = 1}^N {{x_j}(0)},
\bar v = \frac{1}{N}\sum\nolimits_{j = 1}^N {{v_j}(0)}.\]
Consensus of the second-order MAS (1) is said to be reached asymptotically if
\begin{equation}
\begin{aligned}
&\mathop {\lim }\limits_{k \to \infty } {x_i}(k) = \bar x +  \bar v \mathop {\lim }\limits_{k \to \infty }k\tau ,\\
&\mathop {\lim }\limits_{k \to \infty } {v_i}(k) = \bar v,
\,\,i = 1, \ldots ,N
\end{aligned}
\end{equation}
holds for any initial state $x_i(0),v_i(0)$.
\end{definition}

The purpose of this paper is to utilize the agent's memory to accelerate the consensus of second-order MASs.

\section{Consensus Analysis}

In this section, the consensus problem of the second-order MAS is transformed into the simultaneous stabilization problem of $N\!-\!1$ subsystems, and the necessary and sufficient condition for consensus is given.

Let $\bm{x}(k)\in \mathbb{R}^N$ and $\bm{v}(k)\in \mathbb{R}^N$ be the column stack of $x_i(k)$ and $v_i(k)$, respectively.
The compact form of the system can be written as
\begin{equation}
\begin{aligned}
\bm{x}(k + 1) = &\;\bm{x}(k) + \tau \bm{v}(k),\\
\bm{v}(k + 1) = &\;(I-\tau\varepsilon_2 \mathcal{L})\bm{v}(k) -\tau \varepsilon_1\mathcal{L}\bm{x}(k)
                \\&+ \tau\sum\limits_{m = 0}^M {{\theta _m}\bm{v}(k - m)}.
\end{aligned}
\end{equation}

\begin{lemma}
The consensus of the second-order MAS (4) is reached only if $\sum\limits_{m = 0}^M {{\theta _m} = 0}$.
\end{lemma}

\begin{proof}
When ${k \!\to\! \infty }$, it follows from (4) that
\begin{equation}
\bar v \bm{1} = (I - \tau {\varepsilon _2}\mathcal{L})\bar v\bm{1} - \tau {\varepsilon _1}\mathcal{L}
(\bar x\bm{1}+ \bar v\bm{1}\mathop {\lim }\limits_{k \to \infty }k\tau   ) + \tau \sum\limits_{m = 0}^M {{\theta _m}\bar v\bm{1}}.
\end{equation}
Note that $\mathcal{L}\bm{1} = \bm{0}$, then (5) becomes
\[\bar v\bm{1} = \bar v\bm{1} + \tau \sum\limits_{m = 0}^M {{\theta _m}\bar v\bm{1}}. \]
Thus, the condition $\sum\limits_{m = 0}^M {{\theta _m} = 0}$ is required to ensure that consensus can be reached.
\end{proof}

According to the graph Fourier transform \cite{6494675},
\[{\tilde x}_i(k) = {\bm{w}_i^T}\bm{x}(k), \,\, {\tilde v}_i(k) = {\bm{w}_i^T}\bm{v}(k),i=1,\ldots,N,\]
the agent's state in the graph spectrum domain has the iterative form
\begin{equation}
\begin{aligned}
{\tilde x}_i(k + 1) = &\;{\tilde x}_i(k) + \tau {\tilde v}_i(k),\\
{\tilde v}_i(k + 1) = &\;(1-\tau\varepsilon_2 \lambda_i){\tilde v}_i(k) -\tau \varepsilon_1\lambda_i {\tilde x}_i(k)
                      \\&\;+ \tau\sum\limits_{m = 0}^M {{\theta _m}{\tilde v}_i(k - m)}, i=1,\ldots,N.
\end{aligned}
\end{equation}

\begin{lemma}
Consider the second-order MAS (4) on a connected network $\mathcal{G}$ with $\sum\limits_{m = 0}^M {{\theta _m} = 0}$.
Then consensus is achieved if and only if
\[\mathop {\lim }\limits_{k \to \infty } {{\tilde x}_i}(k) = 0, \,\,
\mathop {\lim }\limits_{k \to \infty } {{\tilde v}_i}(k) = 0\]
holds for any $i=2,3,\ldots,N$.
\end{lemma}

\begin{proof}
For a connected graph, $\lambda_1=0$ and
\[{{\tilde v}_1}(k + 1) = {{\tilde v}_1}(k) + \tau \sum\limits_{m = 0}^M {{\theta _m}} {{\tilde v}_1}(k - m).\]
It follows from $\sum\limits_{m = 0}^M {{\theta _m} = 0}$ and ${\tilde v_1}(-M) =  \cdots  = {\tilde v_1}( - 1) = {\tilde v_1}(0)$ that
${{\tilde v}_1}(k) = {{\tilde v}_1}(0)$ holds for all $k\ge 0$, and
\[{\tilde x}_1(k) = {\tilde x}_1(0) +  k\tau {\tilde v}_1(0).\]
Then we have
\begin{equation} \nonumber
\begin{aligned}
&{\bm{w}_1}{\tilde x}_1(k) = \frac{1}{N}{\bm{1}\bm{1}^T}\bm{x}(0) + k\tau \frac{1}{N}{\bm{1}\bm{1}^T}\bm{v}(0) = \bar x\bm{1} + k\tau \bar v\bm{1},\\
&{\bm{w}_1}\tilde v_1(k) = \frac{1}{N}{\bm{1}\bm{1}^T}\bm{v}(0) = \bar v\bm{1}.
\end{aligned}
\end{equation}
The final state can be written as
\begin{equation}
\begin{array}{l}
\mathop {\lim }\limits_{k \to \infty } \bm{x}(k) = \bar x\bm{1} + \bar v\bm{1} \mathop {\lim }\limits_{k \to \infty } k\tau   +  \mathop {\lim }\limits_{k \to \infty } \sum\limits_{i = 2}^N {{\bm{w}_i}{{\tilde x}_i}(k)} ,\\
\mathop {\lim }\limits_{k \to \infty } \bm{v}(k) = \bar v\bm{1} + \mathop {\lim }\limits_{k \to \infty }\sum\limits_{i = 2}^N {{\bm{w}_i}{{\tilde v}_i}(k)} .
\end{array}
\end{equation}
Substituting
$\mathop {\lim }\limits_{k \to \infty } {{\tilde x}_i}(k) = 0,
\mathop {\lim }\limits_{k \to \infty } {{\tilde v}_i}(k) = 0,i=2,\ldots,N$
into (7), the sufficiency is proved directly.
Suppose that there is a scalar $s\!\in\!\{2,\ldots,N\}$ that satisfies $\mathop {\lim }\limits_{k \to \infty } {{\tilde x}_s}(k)\!=\! \delta_1  \!\ne\! 0$ or $\mathop {\lim }\limits_{k \to \infty } {{\tilde v}_s}(k)\!=\! \delta_2  \!\ne\! 0$.
Since $\left\langle {{\bm{w}_i},{\bm{w}_j}} \right\rangle  = 0$ holds for any $i\ne j$, then
$\mathop {\lim }\limits_{k \to \infty } \bm{x}(k)- \bar x \bm{1} - \bar v \mathop {\lim }\limits_{k \to \infty } k\tau \bm{1} \!\ne\! \bm{0}$, or
$\mathop {\lim }\limits_{k \to \infty } \bm{v}(k)- \bar v \bm{1} \!\ne\! \bm{0}$.
This contradiction proves the necessity.
\end{proof}

Denote $\bm{y}_i(k)=[{{\tilde x}_i}(k),{{\tilde v}_i}(k),{{\tilde v}_i}(k-1),\cdots,{{\tilde v}_i}(k-M)]^T$.
The consensus problem can be converted to the simultaneous stabilization problem of $N\!-\!1$ systems of $M+2$ dimensions:
\begin{equation}
{\bm{y}_i}(k) = \Phi ({\lambda _i}){\bm{y}_i}(k - 1),i = 2, \ldots ,N,
\end{equation}
where
\begin{equation}
\Phi ({\lambda _i}) = \left[ {\begin{array}{*{20}{c}}
1&\tau &0& \,\,\cdots\,\, &0\\
{ - \tau {\varepsilon _1}{\lambda _i}}&{1\! -\! \tau {\varepsilon _2}{\lambda _i} \!+\! \tau {\theta _0}}&{\tau {\theta _1}}& \,\,\cdots\,\, &{\tau {\theta _M}}\\
0&1&0& \,\,\cdots\,\, &0\\
 \vdots & \ddots & \ddots & \,\,\ddots\,\, & \vdots \\
0& \cdots &0&\,\,1\,\,&0
\end{array}} \right].
\end{equation}

Consensus can be achieved if and only if the eigenvalues of $\Phi ({\lambda _i})$ are all within the unit circle.
The closer the eigenvalues of $\Phi ({\lambda _i})$ is to the origin, the faster the system (4) achieves consensus.
Thus, we define the consensus convergence rate as \cite{XIAO200465,4627467,9763023}
\begin{equation}
{r_M} = \mathop {\max }\limits_{i = 2, \ldots N} {\rho}(\Phi(\lambda_i) ),
\end{equation}
where ${\rho } \left( \cdot\right)$ denotes the spectral radius.

\begin{remark}
The state of system (8) satisfies
\[
\mathop {\lim }\limits_{k \to \infty }\! \left\| {{\bm{y}_i}(k)} \right\| \! \approx \!  \mathop {\lim }\limits_{k \to \infty } \rho {(\Phi ({\lambda _i}))^k}\! \left\| {{\bm{y}_i}(0)} \right\| \! \le \! \mathop {\lim }\limits_{k \to \infty } ({r_M})^k \! \left\| {{\bm{y}_i}(0)} \right\|
\]
 for all $i=2,\ldots,N$.
It means that the consensus error
\[\bm{e}(k) = \left[ {\begin{array}{*{20}{c}}
{\bm{x}(k) - (\bar x + k\tau \bar v)\bm{1}}\\
{\bm{v}(k) - \bar v\bm{1}}
\end{array}} \right]\in \mathbb{R}^{2N}\]
is bounded by $r_{M}$, that is, $\left\Vert e(k)\right\Vert
=O((r_{M})^{k})$ for $k$ large enough.
The condition $r_M<1$ ensures that system (8) converges to zero and system (4) achieves consensus.
The smaller the convergence rate $r_{M}$ is, the faster the consensus error $\bm{e}(k)$ converges.
\end{remark}

\begin{lemma}
Let $\Phi(\lambda_i)$ be defined by (9).
The characteristic polynomial of $\Phi(\lambda_i)$ is given by
\begin{equation}
d(z,\lambda_i ) = {z^M}[(z - 1)(z - 1 + \tau {\varepsilon _2}{\lambda _i} - \sum\limits_{m = 0}^M {\tau {\theta _m}{z^{ - m}}} ) + {\tau ^2}{\varepsilon _1}{\lambda _i}].
\end{equation}
\end{lemma}

\begin{proof}
Denote
\begin{equation} \nonumber
\begin{aligned}
&{Q_{1}} = \left[ {\begin{array}{*{20}{c}}
{z - 1}&{ - \tau }\\
{\tau {\varepsilon _1}{\lambda _i}}&{z \!-\! 1 \!+\! \tau {\varepsilon _2}{\lambda _i} \!-\! \tau {\theta _0}}
\end{array}} \right],\\
&{Q_{2}} = \left[ {\begin{array}{*{20}{c}}
0& \cdots &0\\
{ - \tau {\theta _1}}& \cdots &{ - \tau {\theta _M}}
\end{array}} \right],
\\
&{{Q_{3}} = \left[ {\begin{array}{*{20}{c}}
0&{ - 1}\\
0&0\\
 \vdots & \vdots \\
0&0
\end{array}} \right],{Q_{4}} = \left[ {\begin{array}{*{20}{c}}
z&{}&{}&{}\\
{ - 1}&z&{}&{}\\
{}& \ddots & \ddots &{}\\
{}&{}&{ - 1}&z
\end{array}} \right].}
\end{aligned}
\end{equation}
The inverse of $Q_{4}$ is calculated as
\[Q_{4}^{ - 1} = {\left[ {\begin{array}{*{20}{c}}
{{z^{ - 1}}}&{}&{}&{}\\
{{z^{ - 2}}}&{{z^{ - 1}}}&{}&{}\\
 \vdots & \ddots & \ddots &{}\\
{{z^{ - M}}}& \cdots &{{z^{ - 2}}}&{{z^{ - 1}}}
\end{array}} \right]}.\]
According to Lemma 2, the characteristic polynomial of $\Phi ({\lambda _i})$ can be calculated as
\begin{equation} \nonumber
\begin{aligned}
&\det (zI - \Phi ({\lambda _i}))
\\=& \det ({Q_4}) \cdot \det ({Q_1} - {Q_2}Q_4^{ - 1}{Q_3})
\\=& z^M \det \left[ {\begin{array}{*{20}{c}}
{z - 1}&{ - \tau }\\
{\tau {\varepsilon _1}{\lambda _i}}&{z \!-\! 1 \!+\! \tau {\varepsilon _2}{\lambda _i} \!-\! \sum\limits_{m = 0}^M {\tau {\theta _m}{z^{ - m}}} }
\end{array}} \right]
\\=& z^M [(z - 1)(z - 1 + \tau {\varepsilon _2}{\lambda _i} - \sum\limits_{m = 0}^M {\tau {\theta _m}{z^{ - m}}} ) + {\tau ^2}{\varepsilon _1}{\lambda _i}].
\end{aligned}
\end{equation}
It follows that $\det (zI \!-\! \Phi ({\lambda _i})) \!= \!d(z, \lambda_i).$
\end{proof}

Let $\bar z(d(z,\lambda_i ))$ be the maximum modulus root of $d(z,\lambda_i ) = 0$, i.e.,
\begin{equation}
\bar z(d(z,\lambda_i ))=\max \{\left\vert z \right\vert :d(z,\lambda_i )=0\}, i = 2,\ldots,N.
\end{equation}
The following result can be immediately derived by combing (10), Lemma 4 and Lemma 5.

\begin{theorem}
Consider the second-order MAS (4) on a connected network $\mathcal{G}$ with $\sum\limits_{m = 0}^M {{\theta _m} = 0}$.
Let $d(z,\lambda_i )$ and $\bar z(d(z,\lambda_i ))$ be defined by (11) and (12).
Then
\\(i) consensus is achieved if and only if the roots of \[d(z,{\lambda _i})=0,i=2,\ldots,N\] are all within the unit circle;
\\(ii) the convergence rate $r_M$ in (10) can be computed by \[r_M = \max_{i=2,\ldots ,N} \bar z(d(z,\lambda_i )) .\]
\end{theorem}

Theorem 1 establishes the direct link between the consensus convergence rate and the roots of $d(z,\lambda_i )=0$.
Then the accelerated consensus problem of second-order MASs can be converted into the optimization problem:
\begin{equation}
\begin{array}{l}
r_M^* =   \mathop {\min }\limits_{{\varepsilon _1},{\varepsilon _2},{\theta _0}, \ldots ,{\theta _M}} {r_M}
=\mathop {\min }\limits_{{\varepsilon _1},{\varepsilon _2},{\theta _0}, \ldots ,{\theta _M}} \mathop {\max }\limits_{i = 2, \ldots N} \bar z(d(z,\lambda_i )).
\end{array}
\end{equation}
In the next section, we try to find the optimal control parameters ${\varepsilon _1},{\varepsilon _2},{\theta _0}, \ldots ,{\theta _M}$ to minimize the convergence rate $r_M$.

 \section{Accelerated Consensus}

In this section,
explicit formulas of the optimal control parameters and the corresponding convergence rate with one-tap memory are derived.
Then an iterative algorithm is given to derive the control parameters to accelerate the convergence rate for the case of $M$-tap memory.

\subsection{Explicit Formula of the Optimal Convergence Rate with One-tap Memory}

In this subsection,
we give the consensus region of the control parameters, and derive the analytical formulas of the optimal control parameters and the corresponding convergence rate with one-tap memory.

When $M\!=\!1$, the second-order MAS (4) can be written as
\begin{equation}
\begin{aligned}
\bm{x}(k + 1) = &\;\bm{x}(k) + \tau \bm{v}(k),\\
\bm{v}(k + 1) = &\;(I-\tau\varepsilon_2 \mathcal{L})\bm{v}(k) -\tau \varepsilon_1\mathcal{L}\bm{x}(k)
                \\&+ \tau {\theta _0}(\bm{v}(k) - \bm{v}(k - 1)).
\end{aligned}
\end{equation}
Consensus is achieved if and only if the system
\begin{equation} \nonumber
{\bm{y}_i}(k) = \Phi ({\lambda _i}){\bm{y}_i}(k - 1),i = 2, \ldots ,N,
\end{equation}
is stable, where
\begin{equation}
\Phi ({\lambda _i}) = \left[ {\begin{array}{*{20}{c}}
1&\tau &0\\
{ - \tau {\varepsilon _1}{\lambda _i}}&{1\! -\! \tau {\varepsilon _2}{\lambda _i} \!+\! \tau {\theta _0}}&{-\tau {\theta _0}}\\
0&1&0
\end{array}} \right].
\end{equation}
The characteristic polynomial of $\Phi ({\lambda _i})$ is
\begin{equation}
\begin{aligned}
d(z,{\lambda _i}) =& {{z}^3} + (\tau {\varepsilon _2}{\lambda _i} - \tau {\theta _0} - 2){{z}^2}
\\& + ({\tau ^2}{\varepsilon _1}{\lambda _i} - \tau {\varepsilon _2}{\lambda _i} + 1 + 2\tau {\theta _0})z - \tau {\theta _0}.
\end{aligned}
\end{equation}

The following necessary and sufficient condition for consensus with $M=1$ can be immediately derived by  combing the Jury stability criterion and Theorem 1.
\begin{lemma}
The consensus of the second-order MAS (14) is achieved if and only if
\[\left\{ {\begin{array}{*{20}{l}}
{\left| {\tau {\theta _0}} \right| <1,}\\
{{\tau ^2}{\varepsilon _1}{\lambda _i} > 0,}\\
{4 + 4\tau {\theta _0} - 2\tau {\varepsilon _2}{\lambda _i} + {\tau ^2}{\varepsilon _1}{\lambda _i} > 0,}\\
{| {{\tau ^2}{\theta _0}^2 \!-\!1} | < | {{\tau ^2}{\varepsilon _2}{\theta _0}{\lambda _i} \!+\! {\tau ^2}{\varepsilon _1}{\lambda _i} \!-\! \tau {\varepsilon _2}{\lambda _i} \!-\! {\tau ^2}{\theta _0}^2 \!+\! 1} |,}
\end{array}} \right.\]
holds for $i=2,\ldots,N.$
\end{lemma}

Note that inequalities in Lemma 6 yield a three-dimensional consensus region with respect to the parameters $\varepsilon_0,\varepsilon_1,\theta_0$.
Next, we will find the optimal control parameters in this consensus region to minimize the convergence rate.

\begin{theorem}
Consider the second-order MAS (14) on a connected network $\mathcal{G}$.
The optimal consensus convergence rate is
\begin{equation}
r_1^* = \sqrt {1 - \frac{2}{{\sqrt {2{\lambda _N}/{\lambda _2} - 1}  + 1}}} ,
\end{equation}
with the control parameters
\begin{equation}
\begin{aligned}
&{\varepsilon _1^*} = \frac{1}{{{\tau ^2}{\lambda _N}}}(1 - {(r_1^*)^4}),\\
&{\varepsilon _2^*} = \frac{1}{{\tau {\lambda _N}}}({(r_1^*)^4} + {(r_1^*)^2} + 2),\\
&{\theta _0^*} = \frac{1}{\tau }{(r_1^*)^4}.
\end{aligned}
\end{equation}
\end{theorem}

\begin{proof}
Let $z = r\tilde z$ in (16), where $0<r<1$.
Then
\begin{equation}
\begin{aligned}
{d}(r\tilde z,\lambda_i) =\,& {{r}^3}{{\tilde z}^3} + (\tau {\varepsilon _2}{\lambda _i} - \tau {\theta _0} - 2){{r}^2}{{\tilde z}^2}
\\&+ ({\tau ^2}{\varepsilon _1}{\lambda _i} - \tau {\varepsilon _2}{\lambda _i} + 1 + 2\tau {\theta _0}){r}\tilde z - \tau {\theta _0}.
\end{aligned}
\end{equation}
The roots of ${d}(r\tilde z,\lambda_i)=0$ are within the unit circle if and only if the roots of ${d}(z,\lambda_i)=0$ are within the circle with radius $r$.
According to the Jury stability criterion,
the optimization problem of the convergence rate can be written as
\begin{eqnarray}
\!\!\!\!\!\!\!\!&&\min \limits_{{\varepsilon _1}.{\varepsilon _2},{\theta _0}}\quad r
\end{eqnarray}
\begin{subequations}
s.t. \\
\begin{align}
({r} \!- \!\tau {\theta _0}){({r} \!-\! 1)^2} + ({\tau ^2}{\varepsilon _1}{r}+\tau {\varepsilon _2}{{r}^2} - \tau {\varepsilon _2}{r}){\lambda _i} \ge 0,&\\
({r} \!+ \!\tau {\theta _0}){({r} \!+\! 1)^2} + ( {\tau ^2}{\varepsilon _1}{r}- \tau {\varepsilon _2}{{r}^2} \!-\! \tau {\varepsilon _2}{r} ){\lambda _i} \ge 0,&\\
[({{r}^4} - {\tau ^2}\theta _0^2)({{r}^2} + 1) + 2\tau {\theta _0}{{r}^2}({{r}^2} - 1)]\;\;\;\;\;\;\;\;&
\nonumber\\+ ( - \tau {\varepsilon _2}{{r}^4} + {\tau ^2}{\varepsilon _1}{{r}^4} + {\tau ^2}{\varepsilon _2}{\theta _0}{{r}^2}){\lambda _i} \ge 0,&\\
[({{r}^4} + {\tau ^2}\theta _0^2)({{r}^2} - 1) - 2\tau {\theta _0}{{r}^2}({{r}^2} - 1)]\;\;\;\;\;\;\;\;&
\nonumber\\+ (\tau {\varepsilon _2}{{r}^4} - {\tau ^2}{\varepsilon _1}{{r}^4} - {\tau ^2}{\varepsilon _2}{\theta _0}{{r}^2}){\lambda _i} \ge 0,&\\
{r}^3 - \tau{\theta _0}  \ge 0,&\\
{r}^3 + \tau{\theta _0}  \ge 0,&\\
i = 2,\ldots,N .\,& \nonumber
\end{align}
\end{subequations}

Note that the constraints (21a)-(21d) are all linear with respect to $\lambda_i$.
Consider a simple function
\[g(\lambda)= a _1 + a_2 \lambda, \,\,\,\,{\lambda _2} \le  \lambda \le {\lambda _N}.\]
When $a _1\ge 0$, $g(\lambda)\ge 0$ holds if and only if $g(\lambda_N)\ge 0$.
When $a _1\le 0$, $g(\lambda)\ge 0$ holds if and only if $g(\lambda_2)\ge 0$.

Since
$({r} - \tau {\theta _0}){({r} - 1)^2}>({r^3} - \tau {\theta _0}){({r} - 1)^2}\ge 0,$
the inequality (21a) holds if and only if
\begin{equation}
({r} - \tau {\theta _0}){({r} - 1)^2} + ({\tau ^2}{\varepsilon _1}{r}+ \tau {\varepsilon _2}{{r}^2} - \tau {\varepsilon _2}{r} ){\lambda _N} \ge 0.
\end{equation}
Since $({r} + \tau {\theta _0}){({r} - 1)^2}>({r^3} + \tau {\theta _0}){({r} - 1)^2}\ge 0,$
the inequality (21b) holds if and only if
\begin{equation}
({r} + \tau {\theta _0}){({r} + 1)^2} + ( {\tau ^2}{\varepsilon _1}{r}- \tau {\varepsilon _2}{{r}^2} - \tau {\varepsilon _2}{r}){\lambda _N} \ge 0.
\end{equation}
Since $({r^4} - {\tau ^2}\theta _0^2)({r^2} + 1) + 2\tau {\theta _0}{r^2}({r^2} - 1)\ge ({r^4} - {r^6})({r^2} + 1) + 2{r^5}({r^2} - 1)> 0,$
the inequality (21c) holds if and only if
\begin{equation}
\begin{aligned}
&({r^4} - {\tau ^2}\theta _0^2)({r^2} + 1) + 2\tau {\theta _0}{r^2}({r^2} - 1)
\\&+ ( - \tau {\varepsilon _2}{r^4} + {\tau ^2}{\varepsilon _1}{r^4} + {\tau ^2}{\varepsilon _2}{\theta _0}{r^2}){\lambda _N} \ge 0.
\end{aligned}
\end{equation}
Since $({r^4} + {\tau ^2}\theta _0^2)({r^2} - 1) - 2\tau {\theta _0}{r^2}({r^2} - 1)\le ({r^4} + {r^6})({r^2} - 1) - 2{r^5}({r^2} - 1)< 0,$
the inequality (21d) holds if and only if
\begin{equation}
\begin{aligned}
&({{r}^4} + {\tau ^2}\theta _0^2)({{r}^2} - 1) - 2\tau {\theta _0}{{r}^2}({{r}^2} - 1)
\\&+ (\tau {\varepsilon _2}{{r}^4} - {\tau ^2}{\varepsilon _1}{{r}^4} - {\tau ^2}{\varepsilon _2}{\theta _0}{{r}^2}){\lambda _2} \ge 0.
\end{aligned}
\end{equation}
To eliminate parameters $\varepsilon_1$ and $\varepsilon_2$,
we add (24) times $1/\lambda_N$ and (25) times $1/\lambda_2$, and have
\begin{equation}
\begin{aligned}
&\left( {\frac{{{\tau ^2}({r^2} \!-\! 1)}}{{{\lambda _2}}} \!- \!\frac{{{\tau ^2}({r^2} \!+ \!1)}}{{{\lambda _N}}}} \right)\theta _0^2 + 2\tau \left( {\frac{1}{{{\lambda _N}}} \!-\! \frac{1}{{{\lambda _2}}}} \right)({r^4} \!- \!r^2){\theta _0}
\\&\;\;\;+ {r^4}\left( {\frac{{{r^2} \!+ \! 1}}{{{\lambda _N}}} \!+\! \frac{{{r^2} \!- \!1}}{{{\lambda _2}}}} \right) \ge 0.
\end{aligned}
\end{equation}
The parameter $\theta_0$ in (26) has a real solution set if and only if
\begin{equation}
({\lambda _2} - {\lambda _N}){r^4} + 2{\lambda _N}{r^2} + {\lambda _2} - {\lambda _N} \ge 0.
\end{equation}
Let $t = r^2$ in (27), and
\[f(t)=({\lambda _2} - {\lambda _N}){t^2} + 2{\lambda _N}{t} + {\lambda _2} - {\lambda _N}, \;\;t \in (0,1).\]
Since ${\lambda _2} - {\lambda _N}< 0$ and $\frac{{{\lambda _N}}}{{{\lambda _N} - {\lambda _2}}} > 1$,
the solution set of $f(t)\ge 0$ is
\begin{equation} \nonumber
 {\frac{{{\lambda _N} - \sqrt {{\lambda _2}(2{\lambda _N} - {\lambda _2})} }}{{{\lambda _N} - {\lambda _2}}}}\le t <1.
 \end{equation}
 It follows that
\begin{equation} \nonumber
r = \sqrt{t} \ge \sqrt {1 - \frac{2}{{\sqrt {2{\lambda _N}/{\lambda _2} - 1}  + 1}}}  .
  \end{equation}
By solving
\begin{equation} \nonumber
\begin{aligned}
({r} \!- \!\tau {\theta _0}){({r} \!-\! 1)^2} + ({\tau ^2}{\varepsilon _1}{r}+\tau {\varepsilon _2}{{r}^2} - \tau {\varepsilon _2}{r}){\lambda _N} = 0,&\\
({r} \!+ \!\tau {\theta _0}){({r} \!+\! 1)^2} + ( {\tau ^2}{\varepsilon _1}{r}- \tau {\varepsilon _2}{{r}^2} \!-\! \tau {\varepsilon _2}{r} ){\lambda _N} = 0,&\\
({{r}^4} - {\tau ^2}\theta _0^2)({{r}^2} + 1) + 2\tau {\theta _0}{{r}^2}({{r}^2} - 1)\;\;\;\;\;\;\;\;&
\\+ ( - \tau {\varepsilon _2}{{r}^4} + {\tau ^2}{\varepsilon _1}{{r}^4} + {\tau ^2}{\varepsilon _2}{\theta _0}{{r}^2}){\lambda _N}=0,&\\
({{r}^4} + {\tau ^2}\theta _0^2)({{r}^2} - 1) - 2\tau {\theta _0}{{r}^2}({{r}^2} - 1)\;\;\;\;\;\;\;\;&
\\+ (\tau {\varepsilon _2}{{r}^4} - {\tau ^2}{\varepsilon _1}{{r}^4} - {\tau ^2}{\varepsilon _2}{\theta _0}{{r}^2}){\lambda _2}= 0,&
\end{aligned}
\end{equation}
we can get the optimal convergence rate $r_1^*$ as in (17)
with the control parameters as in (18).
It is verified that the remaining constraints (21e) and (21f)
are satisfied.
This completes the proof.
\end{proof}

\begin{remark}
In \cite{eichler2014closed,dai2022fast}, authors have derived the optimal convergence rate
\[r_0^* = \sqrt {1 - \frac{2}{{{\lambda _N}/{\lambda _2} + 1}}} \]
without memory.
Note that
\[r_1^* = \sqrt {1 - \frac{2}{{\sqrt {2{\lambda _N}/{\lambda _2} - 1}  + 1}}} .\]
Since  $\frac{{2{\lambda _N}}}{{{\lambda _2}}} < {\left( {\frac{{{\lambda _N}}}{{{\lambda _2}}}} \right)^2} + 1$,  we have $\sqrt {\frac{{2{\lambda _N}}}{{{\lambda _2}}} - 1}  < \frac{{{\lambda _N}}}{{{\lambda _2}}}$.
It follows that $r_1^*<r_0^*$.
This means that introducing one-tap memory into the control protocol can further accelerate the convergence rate of the second-order MAS.
Moreover,   for the case of ${\lambda _2}/{\lambda _N} \to 0$, we have
\[\frac{{{{(r_1^*)}^2}}}{{{{(r_0^*)}^2}}} = 1 - \frac{{{{\left( {\sqrt {\frac{{2{\lambda _N}}}{{{\lambda _2}}} - 1}  - 1} \right)}^2}}}{{\left( {\sqrt {\frac{{2{\lambda _N}}}{{{\lambda _2}}} - 1}  + 1} \right)\left( {\frac{{{\lambda _N}}}{{{\lambda _2}}} - 1} \right)}} \approx 1 - \sqrt {\frac{{{2\lambda _2}}}{{{\lambda _N}}}} .\]
\end{remark}

\begin{remark}
The optimal convergence rate $r_1^*$ is related to the eigenratio $\lambda_2/\lambda_N$ of the graph Laplacian.
A larger $\lambda_2/\lambda_N$ corresponds to better network connectivity, leading to a smaller value of $r_1^*$.
When ${\lambda _2}/{\lambda _N} \to 1 $, we have $r_1^*\to 0$.
\end{remark}

\begin{remark}
Since the second-order MAS (14) can achieve consensus with the optimal convergence rate $r_1^*$ by applying the control parameters (18), the sampling period $\tau$ determines the overall convergence time.
If $\tau \to 0$, which implies the infinite band-width communication, then $\varepsilon _1^*,\varepsilon _2^*,\theta_0^* \to \infty$ and consensus will be achieved with arbitrarily fast convergence speed.
\end{remark}

\subsection{Iterative Algorithm of Control Parameters for Accelerating the Convergence Rate}

In this subsection, for the case of $M$-tap memory,
an iterative algorithm based on gradient descent is given to derive the control parameters to accelerate the convergence rate.

Recall that we need to design the control parameters $\varepsilon_1,\varepsilon_2$, $\theta_0,\ldots,\theta_M$ to make the convergence rate $r_M$ as small as possible.
For $M=1$, we have given explicit formulas of the optimal convergence rate and control parameters in the previous subsection.
However, when $M > 1$, it is difficult to derive analytical formulas for the convergence rate and control parameters.
Therefore, we try to design an optimization algorithm to obtain numerical solutions for them.

Before giving the algorithm, we introduce some notations.
Denote \[F(\Theta ) = {r_M} = \mathop {\max }\limits_{i \in \{ 2, \ldots N\} } \rho (\Phi ({\lambda _i})),\]
where
$\Theta = [\varepsilon_1,\varepsilon _2,\theta _0, \ldots ,\theta _{M\!-\!1}]\in \mathbb{R}^{1\times (M+2)}$ represents the stack of control parameters.
Then
\[
\nabla F = {[\nabla {F_1}, \ldots ,\nabla {F_{M\!+\!2}}]^T \in \mathbb{R}^{M+2}},
\]
where $\nabla F_j= \frac{{F(\Theta + \bm{\delta}(j)) - F(\Theta)}}{{{\delta}}}$, and
$\bm{\delta}(j) \in \mathbb{R}^{1\times ({M\!+\!2})}$ is a row vector whose elements are $0$ except the $j$-th term is a tiny positive scalar $\delta$.
Next, Algorithm 1 can obtain the control parameters to accelerate the convergence rate.

\begin{algorithm}[h]
	\caption{Iterative algorithm of control parameters for accelerating the convergence rate}
	\label{alg::Gradient}
	\begin{algorithmic}[1]
		\Require
        nonzero eigenvalues of graph Laplacian $\lambda_i$;
        number of iterations $T$;
        sampling period $\tau$;
        a tiny positive scalar $\delta$;
        memory tap $M$;
        the learning rate $\alpha$;
        the initial control parameter $\Theta^{(0)}$
		\Ensure
		$r^*=F{(\Theta^{(T)})}$, $\Theta^*=\Theta^{(T)}$
		\State Calculate $F(\Theta^{(0)})$;
		\State \textbf{for} $j=1$ to $M\!+\!2$ do
        \State ~~~~$\nabla F_j^{(0)}=\frac{{F(\Theta^{(0)} + \bm{\delta}(j)) - F(\Theta^{(0)})}}{{{\delta}}}$;
        \State \textbf{end for}
        \State \textbf{for} $t=1$ to $T$ do
		\State ~~~~$\Theta^{(t)}= \Theta^{(t-1)}- \alpha \nabla F^{(t-1)}$;
        \State ~~~~\textbf{for} $j=1$ to $M\!+\!2$ do
		\State ~~~~~~~~$\nabla F^{(t)}_j= \frac{{F(\Theta^{(t)}) + \bm{\delta}(j)) - F(\Theta^{(t)})}}{{{\delta}}}$;
        \State ~~~~\textbf{end for}
		\State \textbf{end for}
	\end{algorithmic}
\end{algorithm}


 \section{Extended to the Formation Control}

In this section, the accelerated consensus with one-tap memory is extended to the accelerated formation control.
The control parameters to achieve the fastest formation are given.

Consider the discrete-time second-order dynamics (1).
Denote $p_i \in \mathbb{R}, i=1,\ldots,N$ be the desired formation, and
${\Delta _{ij}}  = {p_i} - {p_j}$ be the relative position of the desired formation  \cite{ren2007distributed,ren2011distributed,oh2015survey}.
We introduce ${\Delta _{ij}}$ into the control protocol (2), and design the following formation control protocol:
\begin{equation}
\begin{aligned}
{u_i}(k) = &{\varepsilon _1}\sum\limits_{j \in \mathcal{N}_i} {{a_{ij}}({x_j}(k) - {x_i}(k)}  - \Delta _{ij})
         \\&+ {\varepsilon _2}\sum\limits_{j \in \mathcal{N}_i}  {{a_{ij}}({v_j}(k) - {v_i}(k)} )
         \\&+ {\theta _0}({v_i}(k) - {v_i}(k\! -\! 1)), \,\,i = 1,\ldots,N.
\end{aligned}
\end{equation}


Formation is said to be achieved asymptotically if
\begin{equation} \nonumber
\begin{aligned}
&\mathop {\lim }\limits_{k \to \infty }( {x_j}(k) - {x_i}(k) )=  \Delta _{ij},\\
&\mathop {\lim }\limits_{k \to \infty }( {v_j}(k) - {v_i}(k) )= 0, \,\,i = 1, \ldots ,N
\end{aligned}
\end{equation}
holds for any initial state $x_i(0),v_i(0)$.
The goal of the accelerated formation control is to design control parameters $\varepsilon _1,\varepsilon _2,\theta_0$ so that agents can quickly achieve the formation.

Let ${\xi}_i(k) = x_i(k)-p_i$, $\bm{{\xi}}(k)=[\xi_1(k),\ldots,{\xi}_N(k)]^T\in \mathbb{R}^{N}$,
and $\bm{v}(k)=[v_1(k),\ldots,v_N(k)]^T\in \mathbb{R}^{N}$.
Then we have
\begin{equation}
\begin{aligned}
\bm{\xi}(k + 1) = &\;\bm{\xi}(k) + \tau \bm{v}(k),\\
\bm{v}(k + 1) = &\;(I_N-\tau\varepsilon_2 \mathcal{L})\bm{v}(k) -\tau \varepsilon_1 \mathcal{L}\bm{\xi}(k)
                \\&+ \tau {\theta _0}(\bm{v}(k) - \bm{v}(k - 1)).
\end{aligned}
\end{equation}
Formation is achieved if and only if system (29) achieves consensus.

According to the graph Fourier transform,
\begin{equation} \nonumber
\begin{aligned}
{\tilde \xi}_i(k) = {\bm{w}_i^T}\bm{\xi}(k) ,\,\, {\tilde v}_i(k) = {\bm{w}_i^T}\bm{v}(k) , i =1,\ldots,N,
 \end{aligned}
\end{equation}
the state of system (29) in the graph spectrum domain has the iterative form
\begin{equation} \nonumber
\begin{aligned}
{\tilde \xi}_i(k + 1) = &\;{\tilde \xi}_i(k) + \tau {\tilde v}_i(k),\\
{\tilde v}_i(k + 1) = &\;(1-\tau\varepsilon_2 \lambda_i){\tilde v}_i(k) -\tau \varepsilon_1\lambda_i{\tilde \xi}_i(k)
                      \\&\;+ \tau {\theta _0}(\tilde v_i(k) - \tilde v_i(k - 1)),i =1,\ldots,N.
\end{aligned}
\end{equation}
Similar to Lemma 4, formation is achieved if and only if
\[\mathop {\lim }\limits_{k \to \infty } {{\tilde \xi}_i}(k) = 0, \,\,
\mathop {\lim }\limits_{k \to \infty } {{\tilde v}}_i(k) = 0\]
holds for any $i=2,\ldots,N$.

Let $\bm{X}_i(k)=[{{\tilde \xi}_i}(k),{{\tilde v}_i}(k),{{\tilde v}_i}(k-1)]^T \in \mathbb{R}^{3}$.
Then the problem of the formation control  becomes the simultaneous stabilization problem of $N-1$ systems:
\begin{equation}
{\bm{X}_i}(k) = \Phi ({\lambda _i}) {\bm{X}_i}(k - 1),i = 2, \ldots ,N,
\end{equation}
where $\Phi ({\lambda _i})\in \mathbb{R}^{3\times 3}$ is defined in (15).

Similar to (10), the convergence rate to achieve the desired formation can be defined as
${r_1} = \mathop {\max }\limits_{i = 2, \ldots, N} {\rho}(\Phi(\lambda_i) ).$
According to Theorem 2, the following corollary can be directly obtained.

\begin{corollary}
Consider the second-order MAS (1) under the control protocol (28) on a connected network $\mathcal{G}$.
The optimal convergence rate to achieve formation is given by (17),
and the corresponding control parameters are given by (18).
\end{corollary}

\section{Numerical Examples}

In this section, three examples are used to demonstrate the validity and correctness of the proposed results.

\subsection{Convergence rate with different M}
This example compares the consensus convergence rates of the second-order MAS with different $M$.

Randomly generate a network $\mathcal{G}_1$ with $8$ nodes by using the small-world network model, as shown in Fig. 1(a). Set $\tau = 0.1$.
For the case of $M=0$, the optimal convergence rate is given by $r_0^* = \sqrt {1 - \frac{2}{{{\lambda _N}/{\lambda _2} + 1}}}$,
which has been proposed in \cite{eichler2014closed,dai2022fast}.
For the case of $M=1$, the optimal convergence rate $r_1^*$ is given by (17).
For the case of $M\ge 2$, the convergence rate $r_M^*$ is given by Algorithm 1.
Table I lists the convergence rate $r_M^*$ and control parameters under different memory taps.
Randomly generate the initial state $x_i(0),v_i(0)$ in the interval $[-10,10]$.
The optimal convergence of the consensus error ${\left\| {\bm{e}(k)} \right\|_2}$
under different memory taps is shown in Fig. 2.

It can be observed from Table I and  Fig. 2 that the larger $M$ is, the faster the consensus is achieved.
In addition, the decrease of the convergence rate from $M=0$ to $M=1$ is significant, but the decrease is not so much
when $M> 1$.

\begin{figure}
\setcounter{subfigure}{0}
\centering
\subfigure[The small-world network $\mathcal{G}_1$]{
\includegraphics[scale=0.25]{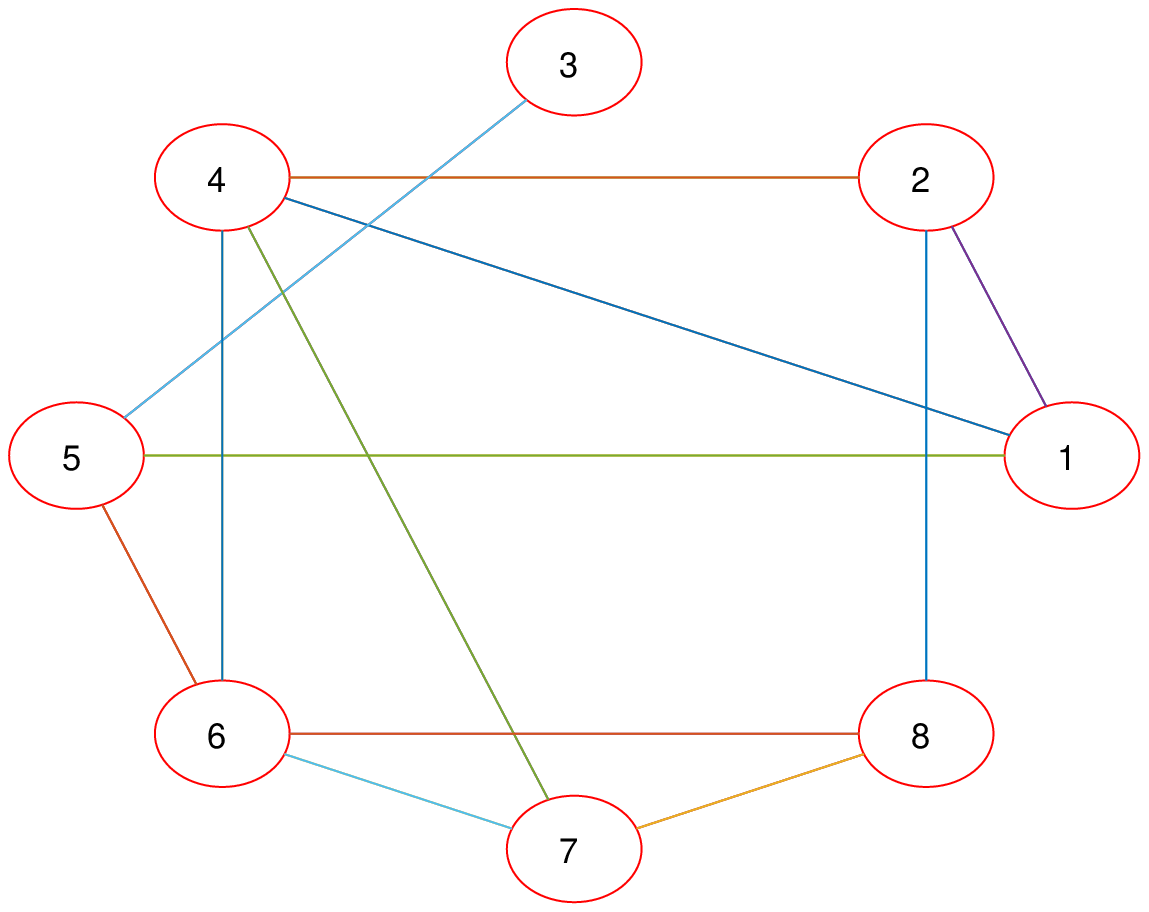}
\label{fig4a}}
\subfigure[The BA scale-free network $\mathcal{G}_2$]{
\includegraphics[scale=0.25]{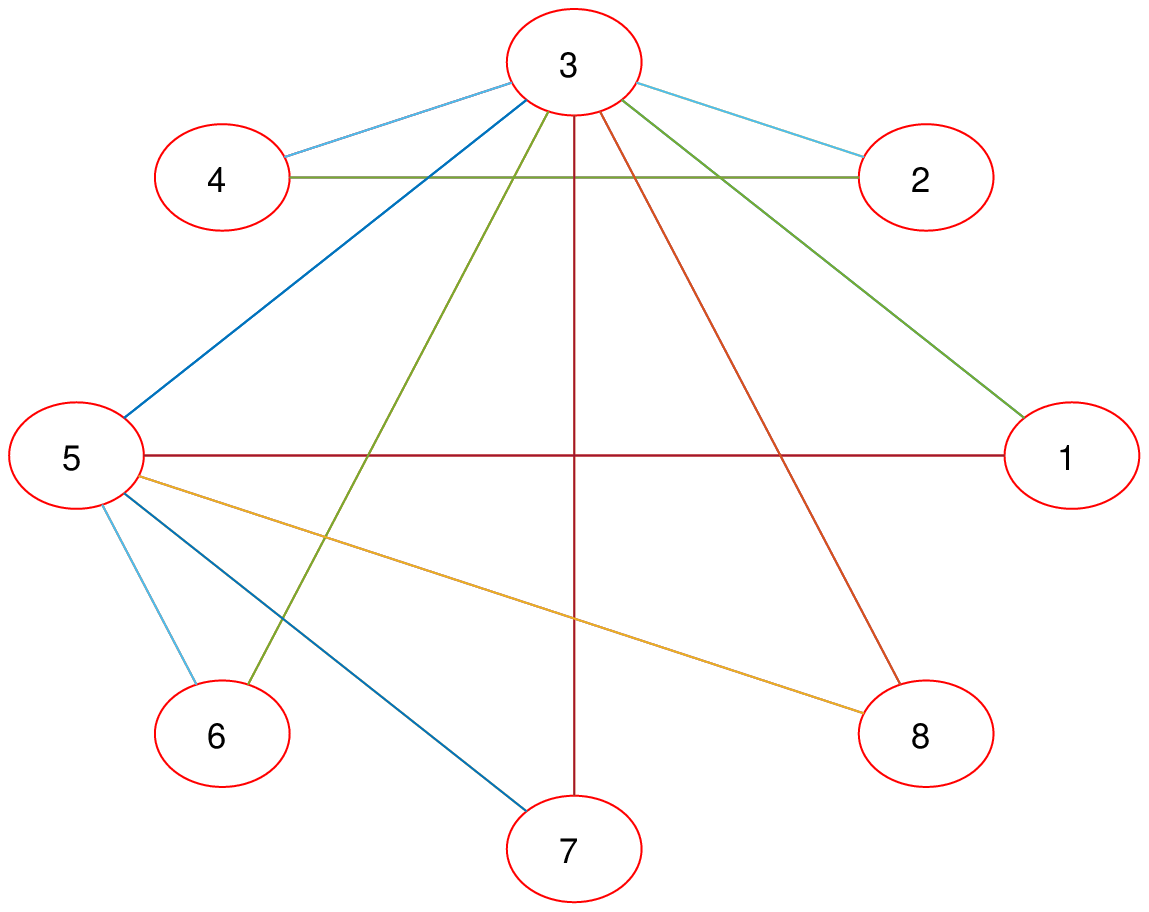}
\label{fig4b}}
  \caption{Networks $\mathcal{G}_1$ and  $\mathcal{G}_2$.}
  \label{fig4}
\end{figure}


\begin{table}[htbp]
\centering
  \caption{The convergence rate $r_M^*$ and corresponding control parameters under different memory taps}
\begin{tabular}{cccccccc}
\toprule
         & $r_M^*$            & $\varepsilon_1^*$   & $\varepsilon_2^*$   & $\theta_0^*$   & $\theta_1^*$ & $\theta_2^*$  & $\theta_3^*$   \\ \midrule
$M\!=\!0$    & \textbf{0.893}            & 3.428                & 3.380          & N/A         & N/A       & N/A        & N/A       \\ \midrule
$M\!=\!1$    & \textbf{0.779}            & 10.674               &5.029           & 3.684       & -3.684    & N/A     & N/A \\ \midrule
$M\!=\!2$    & \textbf{0.768}            & 11.531               &5.202           & 3.674       & -3.936    & 0.262    & N/A \\ \midrule
$M\!=\!3$    & \textbf{0.761}            & 11.405               &5.023           & 3.843       & -3.990    & -0.117   &0.264 \\
\bottomrule
\end{tabular}
\label{tab:addlabel}
\end{table}

\begin{figure}[!htbp]
\centering
\includegraphics[height=6.6cm]{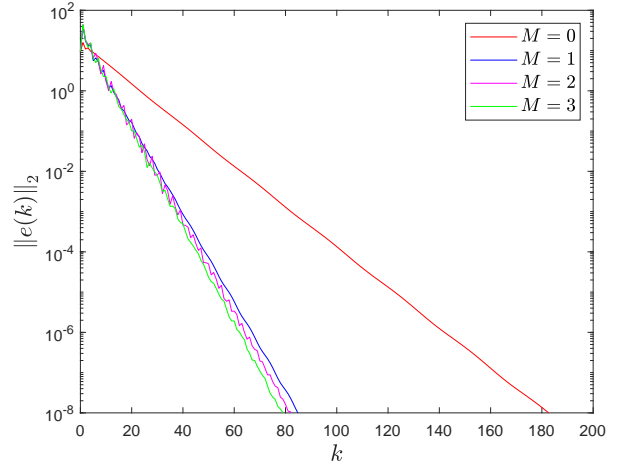}
\caption{Consensus error under different memory taps.}
\label{fig:label}
\end{figure}

\subsection{Convergence rate on different networks}
This example calculates the convergence rate $r_M^*$ on different networks.


Consider the second-order MAS with $8$ nodes on five example networks:
(i) the network $\mathcal{G}_1$;
(ii) the network $\mathcal{G}_2$ generated by a BA scale-free network model, as shown in Fig. 1(b);
(iii) the path network $\mathcal{G}_3$;
(iv) the circle network $\mathcal{G}_4$;
(v) the complete bipartite network $\mathcal{G}_5$ with $3+5$ vertices.
The values of the eigenratio $\lambda_2/\lambda_N$ of the example networks are
$0.113,0.125,0.040,0.146,0.375$, respectively.
Fig. 3 shows the convergence rate $r_M^*$ on sample networks.

It can be observed that a large eigenratio $\lambda_2/\lambda_N$ of the sample networks corresponds to a small $r^*_M$, and the decrease of $r^*_M$ is most significant when $M=0$ to $M=1$.

\begin{figure}[!htbp]
\centering
\includegraphics[height=6.6cm]{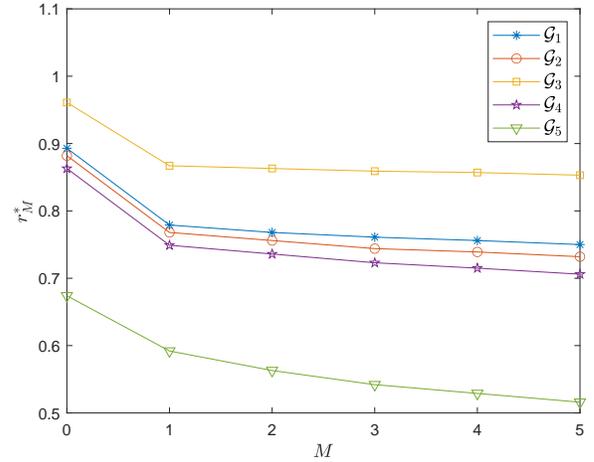}
\caption{$r^*_M$ on sample networks.}
\label{fig:label}
\end{figure}

\subsection{Application in the formation control}
This example applies the accelerated consensus algorithm to the formation control.

Consider the second-order MAS (1) under the control protocol (28) on the network $\mathcal{G}_1$.
Let $n=2$.
The states of the agents are initialized as
\[{\bm{x}_i}(0) = {[i,0]^T},{\bm{v}_i}(0) = {[0.1i,8 - 0.1i]^T},i = 1, \ldots, 8.\]
Set $\tau=0.1$.
Set the agent's desired formation for the first 25 seconds as a circle with radius $3$, that is,
\[{\bm{p}_i} = {[3\cos \frac{\pi }{4}i,3\sin \frac{\pi }{4}i]^T},i = 1, \ldots ,8,\]
for $ 0\le k\le 250$.
Set the agent's desired formation to be a square with side 4 after 25 seconds, that is,
\[\begin{array}{l}
{p_1} = \left[ {\begin{array}{*{20}{c}}
0\\0\end{array}} \right],{p_2} = \left[ {\begin{array}{*{20}{c}}
2\\0\end{array}} \right],{p_3} = \left[ {\begin{array}{*{20}{c}}
4\\0\end{array}} \right],{p_4} = \left[ {\begin{array}{*{20}{c}}
4\\2\end{array}} \right],\\
{p_5} = \left[ {\begin{array}{*{20}{c}}
4\\4\end{array}} \right],{p_6} = \left[ {\begin{array}{*{20}{c}}
2\\4\end{array}} \right],{p_7} = \left[ {\begin{array}{*{20}{c}}
0\\4\end{array}} \right],{p_8} = \left[ {\begin{array}{*{20}{c}}
0\\2\end{array}} \right],
\end{array}\]
for $k> 250$.
According to Corollary 1, the control parameters are given by
$\varepsilon_1^*=10.674,\varepsilon_2^*=5.029,\theta_0^*=3.684.$
The movement trajectory of each agent is shown in Fig. 4.
It can be observed that the agents can achieve the desired formation.

\begin{figure}[!htbp]
\centering
\includegraphics[height=6.6cm]{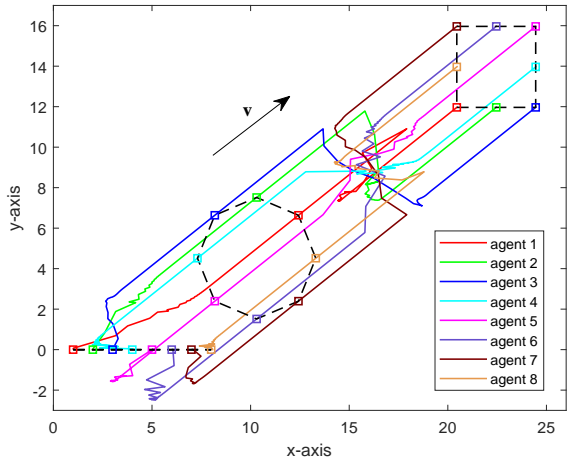}
\caption{Movement trajectory of each agent.}
\label{fig:label}
\end{figure}

\section{Conclusion}

The accelerated consensus problem of second-order multi-agent systems has been studied by introducing the agent's memory into the control protocol.
For the case of one-tap memory, explicit formulas for the optimal convergence rate and control parameters have been derived by using the Jury stability criterion.
It has been proved that the optimal consensus convergence rate with one-tap memory is faster than that without memory.
For the case of $M$-tap memory,
an iterative algorithm based on gradient descent has been given to derive the control parameters to accelerate the convergence rate.
Furthermore, we have extended the accelerated consensus with one-tap memory to the formation control.
Numerical examples have demonstrated the effectiveness of the agent's memory and validated our theoretical results.

\bibliographystyle{IEEEtran}
\bibliography{IEEEabrv,IEEEexample}

\end{document}